\documentclass[12pt]{amsart}
\usepackage[
    margin=1in
]{geometry}

\usepackage{tikz}
\usetikzlibrary{positioning}
\usepackage{amsmath,amssymb,amsxtra,amsthm,hyperref}
\usepackage{tikz}
\usepackage{multicol}
\usepackage{xcolor}
\usetikzlibrary{calc, decorations.markings}
\usepackage{enumitem} 

\newcommand{\bh}[1] {\mathcal{B}(\mathcal{#1})}

\newcommand{\lspan} {\operatorname{span}}

\newcommand{\condref}[1] {\hyperref[cond.#1]{(#1)}}

\newcommand{\cB}{\mathcal{B}}

\newcommand{\cH}{\mathcal{H}}

\newcommand{\cK}{\mathcal{K}}
\newcommand{\cL}{\mathcal{L}}
\newcommand{\cM}{\mathcal{M}}
\newcommand{\cN}{\mathcal{N}}

\newcommand{\bF}{\mathbb{F}}

\newtheorem{theorem}{Theorem}[section]

\newtheorem{lemma}[theorem]{Lemma}

\theoremstyle{definition}

\newtheorem{example}[theorem]{Example}
\newtheorem{remark}[theorem]{Remark}

\numberwithin{equation}{section}

\begin{document}

\title{An Ando-type dilation on right LCM monoids}

\date{\today}


\begin{abstract} We establish an Ando-type dilation theorem for a pair of commuting contractions together with a representation of a right LCM monoid via either the Cartesian or the free product. We prove that if each individual contraction together with the monoid representation has $*$-regular dilation, then they can be dilated to commuting isometries and an isometric representation of the monoid. This extends an earlier result of Barik and Das. 
\end{abstract}

\author{Boyu Li}
\address{Department of Mathematical Sciences, New Mexico State University, Las Cruces, New Mexico, 88003, USA}
\email{boyuli@nmsu.edu}

\author{Mansi Suryawanshi}
\address{Faculty of Mathematics, Technion -- Israel Institute of Technology, Haifa 3200003, Israel}
\email{suryawanshi@campus.technion.ac.il}

\thanks{Boyu Li was supported by an NSF grant (DMS-2350543), and he'd like to thank Kenneth Davidson for showing him the trick of proving Ando's dilation.}
\maketitle

\section{Introduction}

The study of dilation theory originated from the celebrated Sz.-Nagy dilation theorem \cite{Nagy1953, Schaffer1955}, which states that every contractive operator $T$ on a Hilbert space $\cH$ admits an isometric dilation, which is isometry $V$ on $\cK\supset \cH$ such that 
\[P_\cH V^k|_\cH=T^k,\quad \text{for all }\, k\geq 0.\] Extending Sz.-Nagy's dilation to the multivariate setting, Ando \cite{Ando1963} proved that a pair of commuting contractions can be dilated to commuting isometries. However, Ando's dilation fails to extend to three commuting contractions due to various counterexamples \cite{Parrott1970, Varopoulos1974}. As a result, identifying additional conditions that guarantee the existence of isometric dilations has become an active and continuing line of research within multivariable operator theory.

In \cite{Brehmer1961}, Brehmer first introduced the $*$-regular condition for a family of commuting contractions, under which there exist commuting isometric dilations. Recently, Barik and Das \cite{BarikDas2022} showed that for a family of commuting contractions $\{S_1, S_2, T_1,\dots, T_k\}$, suppose $\{S_i, T_1,\dots, T_k\}$ satisfies the Brehmer's $*$-regular conditions for $i=1,2$, then this family of commuting contractions can be dilated to commuting isometries. This generalises earlier results of \cite{GS1997} for $k=1$ and of \cite{BDHS2019} for a certain purity condition. 

In the non-commutative setting, a family of contractions $\{T_1,\dots, T_k\}$ is called a row contraction if $\sum_{i=1}^k T_iT_i^* \leq I$. It is called a row isometry if each $T_i$ is an isometry. It was shown in a series of papers (\cite{Bunce1984, Frazho1982, Popescu1989}) that a row contraction can be dilated to a row isometry. This is now known as the Frazho-Bunce-Popescu dilation theorem. 

The family of row isometries are particularly important in the study of operator algebras because of its close relationship to Cuntz algebras. In his seminal paper \cite{Nica1992}, Nica introduced the notion of isometric Nica-covariant representations of quasi-lattice ordered semigroups. In the special case when the semigroup is the free semigroup $\mathbb{F}_k^+$, an isometric Nica-covariant representation is determined by a family of row isometries. The C*-algebra generated by isometric Nica-covariant representations is now known as the semigroup C*-algebra, which has been generalized to more general semigroups \cite{LS2022, XLi2012}. This spurred a natural question about isometric dilations for contractive semigroup representations. 

In \cite{Li2017}, the first-named author observed a surprising connection between Brehmer's $*$-regular dilation on commuting contractions and Frazho-Bunce-Popescu's isometric dilation on non-commuting row contractions. It turns out they are both special cases of dilations on the family of right-angled Artin monoids $A_\Gamma^+$, also known as the graph product of $\mathbb{N}$. Building upon this discovery, one can extend Brehmer's notion of $*$-regular dilation to representations of right-LCM monoids \cite{Li2019}. A contractive representation of a right-LCM monoid can be dilated to an isometric representation if it satisfies the generalized Brehmer's condition. Moreover, the minimal isometric $*$-regular dilation is an isometric Nica-covariant co-extension. 

Under this new perspective of $*$-regular dilations, let us revisit the result of Barik and Das. The family $\{S_1, S_2, T_1, \dots, T_k\}$ can be viewed as a contractive representation of $\mathbb{N}^2\times\mathbb{N}^k$, where $S_1, S_2$ determines a representation of $S$ of $\mathbb{N}^2$ and $\{T_j\}$ determines a representation of $T$ of $\mathbb{N}^k$. The main result in \cite{BarikDas2022} can be restated as the following: suppose $S_i\times T$, as representations of $\mathbb{N}\times\mathbb{N}^k$ satisfy Brehmer's condition for each $i=1,2$, then one can find an isometric representation of $S\times T$ of $\mathbb{N}^2\times\mathbb{N}^k$. 

In this paper, we prove that we can replace $\mathbb{N}^k$ by any right LCM monoid $P$, and the statement remains true. This establishes a new Ando-type dilation for a contractive representation of $\mathbb{N}^2\times P$. We then considered the free product $\mathbb{N}^2*P$ and obtained a similar Ando-type dilation result using the same technique. Our proof uses a trick of proving Ando's dilation theorem by Davidson and Katsoulis \cite[Corollary 7.11]{DK2011}.

\section{Preliminaries} 

A right LCM monoid $P$ is a unital left-cancellative semigroup such that for all $p,q\in P$, $pP\cap qP$ is either $\emptyset$ or $rP$ for some $r\in P$. For any finite subset $U\subset P$ in a right LCM monoid, the right LCM condition implies that $\bigcap_{p\in U} pP$ is either $\emptyset$ or $rP$ for some $r\in P$. We use $r=\vee U$ to denote the latter case. Throughout this paper, we assume $P$ is countable and discrete.

A contractive representation of a right LCM monoid $P$ is a unital monoid homomorphism $T:P\to\bh{H}$, where each $\|T(p)\|\leq 1$. The representation $T$ is called isometric if each $T(p)$ is isometric. In addition, an isometric representation $T$ is called \emph{Nica-covariant} if for any $p,q\in P$,
\[T(p)T(p)^*T(q)T(q)^* = \begin{cases}
    T(r)T(r)^*, &\text{ if } pP\cap qP=rP, \\
    0, &\text{otherwise.}
\end{cases}\]
Isometric Nica-covariant representation was first introduced by Nica \cite{Nica1992}, which plays a central role in the study of semigroup C*-algebras \cite{CrispLaca2002, LacaRaeburn1996, XLi2012, LS2022}. Right LCM monoids form an important class of monoids in the study of semigroup C*-algebras \cite{ABLS2019, ABCD2021, BLS2017, BLS2018, LL2020, LL2022, Starling2015} due to their comparatively tractable Nica-covariant representations.

An isometric dilation for $T$ is an isometric representation $V:P\to\bh{K}$ on $\cK\supset \cH$, such that $T(p)=P_\cH V(p)|_\cH$ for all $p\in P$. The dilation $V$ is called a \emph{co-extension} of $T$ if $\cH^\perp$ is invariant for $V$, or equivalently, $\cH$ is invariant for $V(p)^*$ for all $p\in P$.

The dilation $V$ is called a \emph{$*$-regular dilation} if for any $p,q\in P$, 
\[P_\cH V(p)^* V(q) |_\cH = \begin{cases}
    T(p^{-1}r)T(q^{-1}r)^*, &\text{ if } pP\cap qP=rP, \\
    0, &\text{ otherwise.}
\end{cases} \]

The concept of $*$-regular dilation originated from Brehmer's study \cite{Brehmer1961} on dilations of commuting contractions, which can be viewed as representations of $\mathbb{N}^k$. 
In a series of papers by the first-named author \cite{Li2016, Li2017, Li2019}, the notion of $*$-regular dilation was extended to general right LCM monoids, where it is shown that having $*$-regular dilation is equivalent to the generalized Brehmer's condition \cite[Theorem 3.9]{Li2019}. 

Notice that when $V$ is a co-extension of $T$, 
\[P_\cH V(a) V(b)^*|_\cH=T(a)T(b)^*,\quad\text{for all } a,b\in P.\]
Therefore, if $V$ is an isometric Nica-covariance co-extension,
\begin{align*}
    P_\cH V(p)^* V(q) |_\cH &= \begin{cases}
    P_\cH V(p^{-1}r)V(q^{-1}r)^*|_\cH, &\text{ if } pP\cap qP=rP, \\
    0, &\text{ otherwise.}
\end{cases} \\
&= \begin{cases}
    T(p^{-1}r)T(q^{-1}r)^*, &\text{ if } pP\cap qP=rP, \\
    0, &\text{ otherwise.}
\end{cases} 
\end{align*}
Hence, any isometric Nica-covariant co-extension must also be a $*$-regular dilation.

A $\ast$-regular dilation $V_0:P\to \cB(\cK_0)$ is called minimal if
\(
\cK_0=\overline{\lspan}\{V_0(p)h:\; p\in P,\ h\in \cH\}.
\)
The minimal $\ast$-regular dilation $V_0:P\to \cB(\cK_0)$ is unique up to unitary equivalence.
A key observation in \cite[Theorem 3.8]{Li2019} is that the minimal $\ast$-regular dilation
$V_0:P\to \cB(\cK_0)$ must be an isometric Nica-covariant co-extension.

This gives us a full characterization of isometric Nica-covariant co-extensions. 

\begin{lemma}\label{lm:minimal.dilation} Let $V:P\to\bh{L}$ be an isometric Nica-covariant co-extension of $T$. Then $\cL$ decomposes as $\cL=\cK_0\oplus\cK_1$ such that $\cK_i,\, i=0, 1$ reduces $V$. Moreover, $V|_{\cK_0}$ is unitarily equivalent to the minimal isometric Nica-covariant dilation (hence $\ast$-regular dilation) $V_0$ of $T$ and $V|_{\cK_1}$ is an isometric Nica-covariant representation. 
\end{lemma}
\begin{proof}
Given an isometric Nica-covariant co-extension $V:P\to\bh{L}$ of $T$, define
\[
\cK_0=\overline{\lspan}\{V(p)h:\ p\in P,\ h\in\cH\}.
\]
It is clear that $\cK_0$ is invariant for $V$. By the Nica-covariance condition,
$V(p)^*V(q)h\in \cK_0$ for all $p,q\in P$ and $h\in\cH$, and hence $\cK_0$ is also
invariant for $V^*$. Therefore, $\cK_0$ reduces $V$, and $\cL=\cK_0\oplus \cK_1$
for some Hilbert space $\cK_1$.

The restriction $V|_{\cK_0}$ is an isometric Nica-covariant co-extension of $T$
satisfying
\(
\cK_0=\overline{\lspan}\{V(p)h:\ p\in P,\ h\in \cH\},
\)
and hence is minimal. Consequently, $V|_{\cK_0}$ is a minimal $\ast$-regular
dilation of $T$, and by uniqueness of minimal $\ast$-regular dilations, it is
unitarily equivalent to $V_0$. Therefore,
\[
V \cong V_0 \oplus X,
\]
where $V_0$ is the minimal $\ast$-regular dilation of $T$ and
$X:=V|_{\cK_1}$ is an isometric Nica-covariant representation, since $V$ is isometric and Nica-covariant.
\end{proof}

Suppose $P, Q$ are two right LCM monoids. 
We use $Q\times P$ to denote their Cartesian product, and $Q*P$ to denote their free product. Both $Q\times P$ and $Q*P$ are also right LCM monoids \cite{FK2009}. 
Given representations $S:Q\to\bh{H}$ and $T:P\to\bh{H}$, we define $S\times T$ to be the representation of $Q\times P$ that maps $(q,p)$ to $S(q)T(p)$. Since $(q,e)$ and $(e,p)$ commute in the Cartesian product, $S\times T$ defines a representation of $Q\times P$ if $S(q)$ commutes with
$T(p)$ for all $q\in Q$ and $p\in P$. 
Similarly, define $S*T$ to be the representation of $Q*P$ that maps $q\in Q$ to $S(q)$ and $p\in P$ to $T(p)$. Since elements of $Q$ and $P$ do not need to commute in $Q*P$, $S*T$ is always a representation of $Q*P$. 

For the rest of this paper, we focus on the case where $Q$ is either $\mathbb{N}$ or $\mathbb{N}^2$. Let $P$ be a right-LCM monoid and $T:P\to\bh{H}$ be a representation. Suppose $S\in\bh{H}$ commutes with all $T(p), p\in P$. One can treat $S$ as a representation of $\mathbb{N}$ by mapping each $k\geq 0$ to $S^k$. Since $S$ commutes with $T$, one can define the representation $S\times T: \mathbb{N}\times P \to\bh{H}$ by
\[S\times T(n,p)=S^nT(p).\]

Now suppose $S_1, S_2$ are two commuting contractions that both commute with $T$. Since $S_1, S_2$ commute, they define a representation $S_1\times S_2$ of $\mathbb{N}^2$ by 
\[(S_1\times S_2)(n,m)=S_1^nS_2^m.\]
From here, we can similarly define the representation $(S_1\times S_2)\times T$ of $\mathbb{N}^2\times P$ by
\[(S_1\times S_2)\times T(n,m,p)=S_1^nS_2^mT(p).\]

Recall, two operators $A, B$ doubly commute if $AB=BA$ and $A^*B=BA^*$. Suppose $B$ is a representation of $P$. We say $A$ and $B$ doubly commute if $A$ and $B(p)$ doubly commute for each $p\in P$. The following Lemma is a well-known fact about isometric Nica-covariant representations of $\mathbb{N}\times P$. 

\begin{lemma} \label{lemma: iso nica for cross} Let $V\in\bh{K}$ be an isometry and $W:P\to\bh{K}$ be an isometric Nica-covariant representation of $P$. Then $V\times W$ is an isometric Nica-covariant representation of $\mathbb{N}\times P$ if and only if $V$ and $W$ doubly commute.
\end{lemma}


Finally, for any contraction $S\in\bh{H}$ and representation $T:P\to\bh{H}$, one can always define $S*T$ to be the representation that maps the elements in $P$ to $T(p)$ and the generator for $\mathbb{N}$ to $S$. The following Lemma characterizes isometric Nica-covariant representations of $\mathbb{N}*P$.

\begin{lemma}\label{lm.free.prod} 
Let $V\in\bh{K}$ be an isometry and $W:P\to\bh{K}$ be an isometric Nica-covariant representation of $P$. Let $P_0$ denote the set of all invertible elements in $P$. Then, $V* W$ is an isometric Nica-covariant representation of $\mathbb{N}* P$ if and only if $V$ and $W(p)$ have orthogonal ranges for all $p\notin P_0$, or equivalently, $VV^*+W(p)W(p)^*\leq I$.
\end{lemma}

\section{Main Results}
The main goal of this paper is the following theorem.

\begin{theorem}\label{thm.main}
Let $P$ be a right LCM monoid and let $(S_1\times S_2)\times T$ be a contractive
representation of $\mathbb{N}^2\times P$ on $\mathcal H$.
Suppose that each $S_i\times T$ has a $*$-regular dilation for $i=1,2$.
Then there exists an isometric dilation $(V_1\times V_2)\times W$
of $(S_1\times S_2)\times T$.
Moreover, for $i=1,2$, the representations $V_i\times W$ can be chosen to be
isometric Nica-covariant co-extensions of $S_i\times T$.
\end{theorem}

\begin{proof}  
Let $A_i\times B_i$ be the minimal $*$-regular dilations of $S_i\times T$, acting on $\mathcal M_i$, for $i=1,2$. 
By \cite{Li2019}, $A_i\times B_i$ are isometric Nica-covariant co-extensions.

By restricting to $P$, the representations
$B_i:P\to \cB(\mathcal M_i)$ are isometric Nica-covariant co-extensions of $T$. Let $W_0:P\to\mathcal B(\mathcal M_0)$ be the minimal isometric Nica-covariant dilation of $T$. 
By Lemma~\ref{lm:minimal.dilation}, for each $i=1,2$, there exists a reducing subspace
$\mathcal M_i^{(0)}\subseteq \mathcal M_i$ such that
$B_i|_{\mathcal M_i^{(0)}} \cong W_0$.
Replacing $B_i$ by a unitarily equivalent representation, we may assume
that $\mathcal M_i^{(0)}=\mathcal M_0$ and hence that
\[
\mathcal M_i=\mathcal M_0\oplus \mathcal N_i,
\qquad
B_i = W_0 \oplus Z_i,
\]
for some isometric representation $Z_i:P\to\mathcal B(\mathcal N_i)$. Since $B_i$ is Nica-covariant, so is $Z_i$.

Now define $\cK=\cM_0\oplus\cN_1\oplus\cN_2$. Since $\cM_i:=\cM_0\oplus\cN_i$, we may identify
\(
\cK \cong \cM_1\oplus \cN_2 \cong \cM_2\oplus \cN_1.
\)
Define
\begin{align*}
    V_1 &= A_1\oplus I_{\cN_2}\in \cB(\cM_1\oplus \cN_2)\cong \cB(\cK), \\
    V_2 &= A_2\oplus I_{\cN_1}\in \cB(\cM_2\oplus \cN_1)\cong \cB(\cK), \\
    W &= W_0\oplus Z_1\oplus Z_2\in \cB(\cM_0\oplus\cN_1\oplus\cN_2)=\cB(\cK).
\end{align*}
Equivalently, $V_1$ acts as $A_1$ on $\cM_0\oplus\cN_1$ and as the identity on $\cN_2$,
while $V_2$ acts as $A_2$ on $\cM_0\oplus\cN_2$ and as the identity on $\cN_1$. Since $W:P\to\cB(\cK)$ is an isometric representation and satisfies
$P_{\cH}W(p)\big|_{\cH}
=
P_{\cH}W_0(p)\big|_{\cH}
=
T(p),
$ for all $p\in P$, it is an isometric dilation of $T$. Since $\cH\subset\cM_0$, $V_i$ is an isometric dilation of $S_i,$ for $i=1,2$. For $i\neq i'$, under the identification $\cK\cong \cM_i\oplus \cN_{i'}$ we have
\[
V_i\times W \;=\; (A_i\times B_i)\ \oplus\ (I_{\cN_{i'}}\times Z_{i'}).
\]
Since the identity representation $I_{\cN_{i'}}$  
doubly commutes with $Z_{i'}$ trivially, it follows that the product representation
$I_{\cN_{i'}}\times Z_{i'}$ satisfies the Nica covariance relations.
As $A_i\times B_i$ is an isometric Nica-covariant co-extension of $S_i\times T$, and the summand $I_{\cN_{i'}}\times Z_{i'}$ is isometric Nica-covariant representation. It follows that  $V_i\times W$ is an isometric Nica-covariant co-extension of $S_i\times T$, for $i=1,2$.

Since $V_i\times W$ is Nica-covariant, $V_i$ and $W$ doubly commute by Lemma~\ref{lemma: iso nica for cross}. 
Consequently, both $V_1V_2$ and $V_2V_1$ doubly commute with $W$. By Lemma~\ref{lemma: iso nica for cross}, both $(V_1V_2)\times W$ and $(V_2V_1)\times W$ are isometric Nica-covariant representations
of $\mathbb N\times P$. Since $V_i\times W$ are co-extensions of $S_i\times T$, $\cH^\perp$ is invariant for $V_i$ and $W$ and thus both $(V_1V_2)\times W$ and $(V_2V_1)\times W$ are isometric Nica-covariant co-extensions of $(S_2S_1)\times T=(S_1S_2)\times T$.

Again, by Lemma~\ref{lm:minimal.dilation}, we can decompose
\(
\mathcal K=\mathcal K_1\oplus\mathcal K_2
\) into reducing subspaces,
such that the restrictions of $(V_1V_2)\times W$ and $(V_2V_1)\times W$ on $\cK_1$ are unitarily equivalent to the minimal isometric Nica-covariant dilation $M\times W_1$ for $(S_1S_2)\times T$, and $\cH\subset\cK_1$. The restrictions of $(V_1V_2)\times W$ and $(V_2V_1)\times W$ on $\cK_2$ are some isometric Nica-covariant representation, and therefore, we may write
\begin{align*}
    (V_1V_2)\times W &= (M\times W_1) \oplus (X\times W_2),  \\
    (V_2V_1)\times W &= (M\times W_1)\oplus (Y\times W_2),
\end{align*}
Here, $X$ and $Y$ are isometries acting on $\mathcal K_2$, while
$W_2:=W|_{\mathcal K_2}$. Moreover, $X\times W_2$ and $Y\times W_2$ are the restrictions of $(V_1V_2)\times W$ and $(V_2V_1)\times W$ on $\cK_2$.

Let $\cK_2^\infty \cong \cK_2 \otimes \ell^2(\mathbb N)$ denote the infinite direct sum
of $\cK_2$, and set
\(
\mathcal L = \mathcal K \oplus \cK_2^\infty \oplus \cK_2^\infty .
\)
Define operators on $\mathcal L$ by
\begin{align*}
    V_1' &= V_1 \oplus X^{\infty} \oplus Y^{\infty}, \\
    V_2' &= V_2 \oplus I_{\cK_2^\infty} \oplus I_{\cK_2^\infty}, \\
    W'   &= W \oplus W_2^{\infty} \oplus W_2^{\infty},
\end{align*}
where $X^{\infty}$ (respectively, $Y^{\infty}$) denotes the infinite ampliation
of $X$ (respectively, $Y$) acting on $\cK_2^\infty$.
Since $V_i\times W$ is an isometric Nica-covariant co-extension of $S_i\times T$,
and since $X\times W_2$, $Y\times W_2$, and $I_{\cK_2}\times W_2$
are isometric Nica-covariant representations of $\mathbb N\times P$,
it follows that $V_i'\times W'$
are isometric Nica-covariant representations of $\mathbb N\times P$.
Moreover, as $\mathcal H \subset \mathcal K \subset \mathcal L$ and the additional
summands in $V_i'\times W'$ act on subspaces orthogonal to $\mathcal H$, each $V_i'\times W'$
is an isometric Nica-covariant co-extension of $S_i\times T$.

With respect to the decomposition
\(
\mathcal L=\mathcal K_1\oplus\mathcal K_2\oplus\mathcal K_2^\infty\oplus\mathcal K_2^\infty,
\)
the operators $V_1'V_2'$ and $V_2'V_1'$ admit the block–diagonal forms
\[
V_1'V_2'=(M\oplus X)\oplus X^\infty\oplus Y^\infty
\quad\text{and}\quad
V_2'V_1'=(M\oplus Y)\oplus X^\infty\oplus Y^\infty.
\]
Let $F\in\cB(\cK_2\oplus\cK_2^\infty\oplus\cK_2^\infty)$ be given by
\[F(k,\xi_1,\xi_2,\dots,\eta_1,\eta_2,\dots)=(\eta_1,k,\xi_1,\xi_2,\dots,\eta_2,\eta_3,\dots).\]
Define $U=I_{\cK_1}\oplus F\in\cB(\cL)$. Then a direct computation shows that, 
\[
U^*(V_2'V_1')U = V_1'V_2'.
\]
By construction, $\mathcal K_1$ reduces $U$ and $U|_{\mathcal K_1}=I_{\mathcal K_1}$.
Since $\mathcal H\subset\mathcal K_1$, it follows that $\mathcal H$ also reduces $U$
and $U|_{\mathcal H}=I_{\mathcal H}$.

Now define
\(
\widetilde V_1 := V_1' U,
\)
and
\(
\widetilde V_2 := U^* V_2'.
\)
Then
\[
\widetilde V_1 \widetilde V_2
= (V_1'U)(U^*V_2')
= V_1'V_2'
= U^*(V_2'V_1')U
= \widetilde V_2 \widetilde V_1.
\]
Hence, $\widetilde V_1$ and $\widetilde V_2$ commute.
Moreover, since $U$ is unitary and $V_1',V_2'$ are isometries, it follows that
$\widetilde V_1$ and $\widetilde V_2$ are also isometries.
Since $\tilde{V}_i$ and $V_i'$ coincide on the reducing subspace $\cK_1$,  $\tilde{V}_i$ are isometric dilations of $S_i$ for $i=1,2$. 

Now, with respect to
\(
\mathcal L=(\mathcal K_1\oplus\mathcal K_2)\oplus \mathcal K_2^\infty\oplus \mathcal K_2^\infty,
\) we have
\[
W'=(W_1\oplus W_2)\oplus W_2^\infty\oplus W_2^\infty .
\]
On $\cK_1^\perp\cong\cK_2\oplus\cK_2^\infty\oplus\cK_2^\infty\cong\cK_2^\infty$, $W'|_{\cK_1^\perp}\cong W_2^\infty$, which obviously doubly commutes with the permutation operator $F$. On $\cK_1$, $U|_{\cK_1}=I_{\cK_1}$ also doubly commute with $W'|_{\cK_1}$. Therefore, $U$ and $W'$ doubly commute.

Now, $V_i'$ and $U$ all doubly commute with $W'$, and therefore, by Lemma~\ref{lemma: iso nica for cross}, $\tilde V_i\times W'$ are isometric Nica-covariant representations. 
{Each $\tilde{V}_i\times W'$ is a co-extension of $ S_i\times T$ because it contains a direct summand of the minimal isometric Nica-covariant dilation $M\times W_1$.} Moreover, since $\tilde{V}_1$ and $\tilde{V}_2$ commute, we have 
$(\tilde{V}_1\times \tilde{V}_2)\times W'$ is an isometric dilation of $(S_1\times S_2)\times T$.
\end{proof}

\begin{remark}
We point out that in the simplest case, when $P=\{e\}$ and $T(e)=I$, 
Theorem~\ref{thm.main} recovers Ando's dilation theorem: two commuting
contractions $S_1,S_2$ admit commuting isometric dilations $V_1,V_2$.
In fact, the main technique used in the proof of
Theorem~\ref{thm.main} is inspired by the proof of Ando's theorem in
\cite[Corollary~7.11]{DK2011}.
\end{remark}

\begin{example}
When $P=\mathbb{N}^k$, a contractive representation $T:P\to\mathcal B(\mathcal H)$
is uniquely determined by a family of $k$ commuting contractions
$\{T_1,\dots,T_k\}$.
In this case, Theorem~\ref{thm.main} asserts that if each 
$\{S_i,T_1,\dots,T_k\}$ admits a $*$regular dilation, then the commuting
contractions $\{S_1,S_2,T_1,\dots,T_k\}$ admit commuting isometric dilations
$\{V_1,V_2,W_1,\dots,W_k\}$.
Moreover, for each $i=1,2$, the family
$\{V_i,W_1,\dots,W_k\}$ can be chosen to be doubly commuting isometric
co-extensions of $\{S_i,T_1,\dots,T_k\}$,
since for isometric representations of $\mathbb N^k$,
Nica-covariance is equivalent to double commutativity of the generators. 
This yields a short proof of the main result in
\cite[Theorem~3.9]{BarikDas2022} {by identifying
$S_1=T_1$, $S_2=T_n$, and $(T_1,\dots,T_k)=(T_2,\dots,T_{n-1})$,
} (see also \cite{BDHS2019}).
\end{example}

As another application of Theorem~\ref{thm.main}, we can obtain a new Ando-type dilation on $\mathbb{N}^2\times\mathbb{F}_k^+$. 

\begin{example}
Let $P=\mathbb{F}_k^+$. A contractive representation $T$ of $P$ is determined by
a family of $k$ (not necessarily commuting) contractions $\{T_1,\dots,T_k\}$.
Let $S$ be a contraction that commutes with all $T_j$.
Then $S\times T$ has a $*$-regular dilation \cite{Li2017} if and only if
$\{T_1,\dots,T_k\}$ is a row contraction and
\[
I-SS^*-\sum_{j=1}^k T_jT_j^*+\sum_{j=1}^k ST_jT_j^*S^*\ge 0.
\]

In this case, Theorem~\ref{thm.main} yields the following. Suppose
$\{T_1,\dots,T_k\}$ is a row contraction and $S_1,S_2$ are commuting contractions
such that $S_iT_j=T_jS_i$ for all $i=1,2$ and all $j$.
Assume that
\[
I-S_iS_i^*-\sum_{j=1}^k T_jT_j^*+\sum_{j=1}^k S_iT_jT_j^*S_i^*\ge 0,
\qquad i=1,2.
\]
Then $S_1$ and $S_2$ dilate to commuting isometries $V_1$ and $V_2$, and
$\{T_j\}$ dilates to a row isometry $\{W_j\}$.
Moreover, for each $i=1,2$, $V_i$ doubly commutes with each $W_j$.

We emphasize that one cannot, in general, expect $V_1$ and $V_2$ to be doubly
commuting. Indeed, arranging this additional property would amount to producing
a Nica-covariant (hence $*$-regular) isometric dilation of
$(S_1\times S_2)\times T$, which is known to be equivalent to a stronger
positivity condition in \cite{Li2017}.
\end{example}

We now consider the right LCM monoid $\mathbb{N}^2*P$. Using the same technique as in the proof of Theorem~\ref{thm.main}, we obtain the following result. 

\begin{theorem}\label{thm.main2} Let $P$ be a discrete right-LCM monoid and let $(S_1\times S_2)\ast T$ be a contractive representation of $\mathbb{N}^2\ast P$ on a separable Hilbert space $\cH$. Suppose $S_i\ast T$ have $*$-regular dilation for $i=1,2$. Then there exists an isometric dilation $(V_1\times V_2)\ast W$ of $(S_1\times S_2)\ast T$. Moreover, $V_i\ast W$ can be chosen to be {isometric} Nica-covariant co-extensions. 
\end{theorem}

\begin{proof} Let $A_i*B_i:\mathbb{N}*P\to\mathcal B(\mathcal M_i)$ denote the $*$-regular, and hence minimal isometric Nica-covariant dilations of $S_i*T$ for $i=1,2$.
For $i=1,2,$ Lemma~\ref{lm:minimal.dilation} guarantees a decomposition
\(
\mathcal M_i=\mathcal M_0\oplus\mathcal N_i
\)
into reducing subspaces such that, with respect to this decomposition,
\[
B_i = W_0 \oplus Z_i,
\]
where $W_0:P\to\mathcal B(\mathcal M_0)$ is the minimal isometric
Nica-covariant dilation of $T$, and
$Z_i:P\to\mathcal B(\mathcal N_i)$ is an isometric Nica-covariant representation.

Define
\(
\mathcal K
=
\mathcal M_0 \oplus \mathcal N_1 \oplus \mathcal N_2 \oplus \ell^2(P)^\infty,
\)
where $\ell^2(P)^\infty$ denotes the infinite ampliation
$\ell^2(P)\otimes \ell^2(\mathbb N)$.
Let
\[
W
=
W_0 \oplus Z_1 \oplus Z_2 \oplus \lambda^\infty
\quad \text{on } \mathcal K.
\]
Here $\ell^2(P)=\overline{\operatorname{span}}\{e_p : p\in P\}$ denotes the Hilbert space
with canonical orthonormal basis $\{e_p\}_{p\in P}$,
$\lambda:P\to\mathcal B(\ell^2(P))$ denotes the left regular representation
defined by $\lambda(p)e_q=e_{pq}$, and $\lambda^\infty$ denotes its infinite amplification.
The representation $\lambda$ is isometric and Nica-covariant.

Let $P_0\subset P$ denote the set of invertible elements of $P$, and set
$\mathcal L_0=\ell^2(P_0)^\infty$.
Since $P$ is a monoid, $1\in P_0$, and hence $P_0\neq\emptyset$.
We identify $\mathcal L_0=\ell^2(P_0)^\infty$ with the closed subspace of
$\ell^2(P)^\infty\cong \ell^2(P)\otimes \ell^2(\mathbb N)$ spanned by
$\{e_p\otimes \delta_n:\; p\in P_0,\ n\in\mathbb N\}$,
where $\{\delta_n\}_{n\in\mathbb N}$ is the canonical orthonormal basis of
$\ell^2(\mathbb N)$.
{It is clear that $\cL_0$ is an infinite-dimensional separable Hilbert space.}
Therefore, we can pick isometries $C_i: \cN_i\oplus\ell^2(P)^\infty\to\cL_0\subset \cN_i\oplus \ell^2(P)^\infty$. Define
\[    V_1 = A_1\oplus C_2, \text{ and } V_2=A_2\oplus C_1.\]
Then $V_1, V_2$ are isometries on $\cK$. 
Moreover, if $p\notin P_0$, then $pP\cap P_0=\varnothing$, and hence
\[
\operatorname{ran}\lambda(p)
=\overline{\operatorname{span}}\{e_{pq}:q\in P\}
\;\perp\;
\ell^2(P_0).
\]
Consequently,
\(
\operatorname{ran}\lambda(p)^{\infty}\perp \cL_0
\), and thus, one can verify that $V_i$ and $W(p)$ have orthogonal ranges for all $p\notin P_0$. 
By Lemma~\ref{lm.free.prod}, both $V_i*W$ are isometric Nica-covariant co-extensions of $S_i*T$. 

From here, we can follow the same construction as in the proof of Theorem~\ref{thm.main}. One can check that both $(V_1V_2)*W$ and $(V_2V_1)*W$ are isometric Nica-covariant co-extensions of $(S_1S_2)*T=(S_2S_1)*T$. By the existence of a minimal isometric Nica-covariant dilation $M*W_1$ for $(S_1S_2)*T$, we can decompose $\cK=\cK_1\oplus\cK_2$ using Lemma \ref{lm:minimal.dilation}, under which 
\begin{align*}
    (V_1V_2)* W & =(M* W_1) \oplus (X* W_2)  \\
    (V_2V_1)* W &= (M* W_1)\oplus (Y* W_2)
\end{align*} 
Define $\cK_2^\infty \cong \cK_2\otimes \ell^2(\mathbb{N})$ and set
\(
\cL=\cK \oplus \cK_2^\infty \oplus \cK_2^\infty .
\)
Define operators on $\cL$ by
\begin{align*}
    V_1' &= V_1 \oplus X^{\infty}\oplus Y^\infty, \\
    V_2' &= V_2 \oplus I_{\cK_2^\infty} \oplus I_{\cK_2^\infty}, \\
    W'   &= W \oplus W_2^\infty\oplus W_2^\infty,
\end{align*}
where $X^\infty$ (resp.\ $Y^\infty$) denotes the infinite ampliation of $X$ (resp.\ $Y$)
acting on the corresponding copy of $\cK_2^\infty$.
By the same trick as in the proof of Theorem~\ref{thm.main}, there exists a unitary
$U\in\cB(\cL)$ such that
\[
U^*(V_2'V_1')U = V_1'V_2'.
\]
Moreover, $U|_{\cK_1}=I_{\cK_1}$ and $U$ doubly commutes with $W'$. 

Now $\tilde{V}_1=V_1' U$ and $\tilde{V}_2=U^*V_2'$ are commuting isometries and thus $(\tilde{V}_1\times \tilde{V}_2)*W'$ is an isometric dilaion of $(S_1\times S_2)*T$. Moreover, since $U$ and $W'$ doubly commute, we have $\tilde{V}_i^* W'(p)=0$ for all $p\notin P_0$. Therefore, both $\tilde{V_i}*W'$ are isometric Nica-covariant co-extensions of $S_i*T$. 
\end{proof}

\begin{remark}
We remark that the infinite amplification $\ell^2(P)^\infty$ in the proof is needed
to make room for the ranges of the operators $C_i$, since each $V_i$ is required to
have range orthogonal to that of $W$.
\end{remark}

\begin{remark} Recall, from \cite{Li2019}, $T:P\to\bh{H}$ has $*$-regular dilation if and only if for all finite $F\subset P$, 
\[Z_F=\sum_{U\subset F} (-1)^{|U|} T_{U} T_{U}^* \geq 0.\]
Here, \[T_UT_U^*=\begin{cases}
T_rT_r^*, &\text{ if } \cap_{p\in P} pP=rP, \\
0, &\text{ if } \cap_{p\in P} pP=\emptyset.
\end{cases}\]
Now, for a contraction $S\in \bh{H}$, to verify $S\times T$ and $S*T$ have $*$-regular dilations on $\mathbb{N}\times P$ and $\mathbb{N}*P$, instead of considering all finite subsets $F$, the positivity of a much smaller collection of $Z_F$ is needed. One can apply \cite[Theorem 6.7]{Li2019} to obtain the following characterizations.
\begin{enumerate}
    \item For $S\times T$ on $\mathbb{N}\times P$, $S\times T$ has $*$-regular dilation if for all finite $F\subset P$, 
    \[Z_F\geq 0, \text{ and } SZ_FS^*\leq Z_F.\]
    \item For $S* T$ on $\mathbb{N}*P$, $S*T$ has $*$-regular dilation if for all finite $F\subset P$, 
    \[Z_F\geq 0, \text{ and } Z_F-SS^*\geq 0.\]
\end{enumerate}
\end{remark}

Given a simple graph $\Gamma$ on $n$ vertices $\{v_1,\dots,v_n\}$, recall the right-angled Artin monoid $A_\Gamma^+$ is the monoid 
\[A_\Gamma^+=\langle e_1,\dots,e_n: e_ie_j=e_je_i \text{ if }v_iv_j\text{ is an edge in }\Gamma\rangle.\]
For a vertex $v$, use $\Gamma/\{v\}$ to denote the subgraph by removing $v$. Our main results establish the following Ando-type dilation theorem on a class of right-angled Artin monoids. 

\begin{example} Let $\Gamma$ be a simple graph on $n$ vertices where $v_1v_2$ is an edge. Furthermore, suppose $v_1, v_2$ are either adjacent to all other vertices or not adjacent to any other vertices. Under this assumption, suppose a contractive representation $T:A_\Gamma^+\to\bh{H}$ satisfies that for $i=1,2$, $T|_{A_{\Gamma/\{v_i\}}^+}$ have $*$-regular dilations (as characterized by \cite{Li2017, Li2025}). Then, by Theorem~\ref{thm.main} and \ref{thm.main2}, we know $T$ must have an isometric co-extension $V$. Moreover, $V|_{A_{\Gamma/\{v_i\}}^+}$ are isometric Nica-covarian co-extensions for both $i=1,2$. 
\end{example}

As a special case where the right-angled Artin monoid is
\[A_\Gamma^+=\langle e_1,e_2,f_1,\dots,f_k:e_1e_2=e_2e_1\rangle\cong\mathbb{N}^2*\bF_k^+,\]
we obtain the following dilation theorem on $\mathbb{N}^2*\bF_k^+$.

\begin{example} 
Consider the special case where $P=\mathbb{F}_k^+$. Let
$\{S_1,S_2,T_1,\dots,T_k\}$ be contractions with $S_1S_2=S_2S_1$.
Assume that, for each $i=1,2$, the tuple $(S_i,T_1,\dots,T_k)$ is a row contraction.
Then, by Theorem~\ref{thm.main2}, there exist row isometric dilations
$(V_i,W_1,\dots,W_k)$ of $(S_i,T_1,\dots,T_k)$ such that $V_1$ and $V_2$ commute.

Comparing this new dilation result with other known results in the literature, the row contraction conditions are weaker than the $*$-regular dilation conditions derived in \cite{Li2017}. While the existence of row isometric dilation for each pair $(S_i, T)$ individually follows from Frazho-Bunce-Popescu's dilation of row contractions, it is not clear that $W_1$ and $W_2$ can be chosen to commute with each other. We would also like to point out that a result by Op\v{e}la \cite{Opela2006} states that the family $\{S_1,S_2,T_j: 1\leq j\leq k\}$ has isometric dilations (where $S_i$ are dilated to commuting isometries). However, the row isometric condition is not guaranteed. 
\end{example}


\begin{thebibliography}{aHNSY21}

\bibitem[ABLS19]{ABLS2019}
Zahra Afsar, Nathan Brownlowe, Nadia~S. Larsen, and Nicolai Stammeier.
\newblock Equilibrium states on right {LCM} semigroup {$C^*$}-algebras.
\newblock {\em Int. Math. Res. Not. IMRN}, (6):1642--1698, 2019.

\bibitem[aHNSY21]{ABCD2021}
Astrid an~Huef, Brita Nucinkis, Camila~F. Sehnem, and Dilian Yang.
\newblock Nuclearity of semigroup {$\rm C^\ast$}-algebras.
\newblock {\em J. Funct. Anal.}, 280(2):Paper No. 108793, 46, 2021.

\bibitem[And63]{Ando1963}
T.~And{\^o}.
\newblock On a pair of commutative contractions.
\newblock {\em Acta Sci. Math. (Szeged)}, 24:88--90, 1963.

\bibitem[BD22]{BarikDas2022}
Sibaprasad Barik and B.~Krishna Das.
\newblock Isometric dilations of commuting contractions and {B}rehmer
  positivity.
\newblock {\em Complex Anal. Oper. Theory}, 16(5):Paper No. 69, 25, 2022.

\bibitem[BDHS19]{BDHS2019}
Sibaprasad Barik, B.~Krishna Das, Kalpesh~J. Haria, and Jaydeb Sarkar.
\newblock Isometric dilations and von {N}eumann inequality for a class of
  tuples in the polydisc.
\newblock {\em Trans. Amer. Math. Soc.}, 372(2):1429--1450, 2019.

\bibitem[BLS17]{BLS2017}
Nathan Brownlowe, Nadia~S. Larsen, and Nicolai Stammeier.
\newblock On {$C^*$}-algebras associated to right {LCM} semigroups.
\newblock {\em Trans. Amer. Math. Soc.}, 369(1):31--68, 2017.

\bibitem[BLS18]{BLS2018}
Nathan Brownlowe, Nadia~S. Larsen, and Nicolai Stammeier.
\newblock {$C^*$}-algebras of algebraic dynamical systems and right {LCM}
  semigroups.
\newblock {\em Indiana Univ. Math. J.}, 67(6):2453--2486, 2018.

\bibitem[Bre61]{Brehmer1961}
S.~Brehmer.
\newblock \"{U}ber vetauschbare {K}ontraktionen des {H}ilbertschen {R}aumes.
\newblock {\em Acta Sci. Math. Szeged}, 22:106--111, 1961.

\bibitem[Bun84]{Bunce1984}
John~W. Bunce.
\newblock Models for {$n$}-tuples of noncommuting operators.
\newblock {\em J. Funct. Anal.}, 57(1):21--30, 1984.

\bibitem[CL02]{CrispLaca2002}
John Crisp and Marcelo Laca.
\newblock On the {T}oeplitz algebras of right-angled and finite-type {A}rtin
  groups.
\newblock {\em J. Aust. Math. Soc.}, 72(2):223--245, 2002.

\bibitem[DK11]{DK2011}
Kenneth~R. Davidson and Elias~G. Katsoulis.
\newblock Dilation theory, commutant lifting and semicrossed products.
\newblock {\em Doc. Math.}, 16:781--868, 2011.

\bibitem[FK09]{FK2009}
John Fountain and Mark Kambites.
\newblock Graph products of right cancellative monoids.
\newblock {\em J. Aust. Math. Soc.}, 87(2):227--252, 2009.

\bibitem[Fra82]{Frazho1982}
Arthur~E. Frazho.
\newblock Models for noncommuting operators.
\newblock {\em J. Funct. Anal.}, 48(1):1--11, 1982.

\bibitem[GS97]{GS1997}
Dumitru Ga\c{s}par and Nicolae Suciu.
\newblock On the intertwinings of regular dilations.
\newblock volume~66, pages 105--121. 1997.
\newblock Volume dedicated to the memory of W\l odzimierz Mlak.

\bibitem[Li12]{XLi2012}
Xin Li.
\newblock Semigroup {${\rm C}^*$}-algebras and amenability of semigroups.
\newblock {\em J. Funct. Anal.}, 262(10):4302--4340, 2012.

\bibitem[Li16]{Li2016}
Boyu Li.
\newblock Regular representations of lattice ordered semigroups.
\newblock {\em J. Operator Theory}, 76(1):33--56, 2016.

\bibitem[Li17]{Li2017}
Boyu Li.
\newblock Regular dilation on graph products of {$\Bbb{N}$}.
\newblock {\em J. Funct. Anal.}, 273(2):799--835, 2017.

\bibitem[Li19]{Li2019}
Boyu Li.
\newblock Regular dilation and {N}ica-covariant representation on right {LCM}
  semigroups.
\newblock {\em Integral Equations Operator Theory}, 91(4):Paper No. 36, 35,
  2019.

\bibitem[Li25]{Li2025}
Boyu Li.
\newblock Poisson transforms on right-angled artin monoids.
\newblock {\em arXiv preprint arXiv:2504.00127}, 2025.

\bibitem[LL20]{LL2020}
Marcelo Laca and Boyu Li.
\newblock Amenability and functoriality of right-{LCM} semigroup {${\rm
  C}^*$}-algebras.
\newblock {\em Proc. Amer. Math. Soc.}, 148(12):5209--5224, 2020.

\bibitem[LL22]{LL2022}
Marcelo Laca and Boyu Li.
\newblock Dilation theory for right {LCM} semigroup dynamical systems.
\newblock {\em J. Math. Anal. Appl.}, 505(2):Paper No. 125586, 37, 2022.

\bibitem[LR96]{LacaRaeburn1996}
Marcelo Laca and Iain Raeburn.
\newblock Semigroup crossed products and the {T}oeplitz algebras of nonabelian
  groups.
\newblock {\em J. Funct. Anal.}, 139(2):415--440, 1996.

\bibitem[LS22]{LS2022}
Marcelo Laca and Camila Sehnem.
\newblock Toeplitz algebras of semigroups.
\newblock {\em Trans. Amer. Math. Soc.}, 375(10):7443--7507, 2022.

\bibitem[Nic92]{Nica1992}
A.~Nica.
\newblock {$C^*$}-algebras generated by isometries and {W}iener-{H}opf
  operators.
\newblock {\em J. Operator Theory}, 27(1):17--52, 1992.

\bibitem[Op{\v{e}}06]{Opela2006}
David Op{\v{e}}la.
\newblock A generalization of {A}nd\^{o}'s theorem and {P}arrott's example.
\newblock {\em Proc. Amer. Math. Soc.}, 134(9):2703--2710, 2006.

\bibitem[Par70]{Parrott1970}
Stephen Parrott.
\newblock Unitary dilations for commuting contractions.
\newblock {\em Pacific J. Math.}, 34:481--490, 1970.

\bibitem[Pop89]{Popescu1989}
Gelu Popescu.
\newblock Isometric dilations for infinite sequences of noncommuting operators.
\newblock {\em Trans. Amer. Math. Soc.}, 316(2):523--536, 1989.

\bibitem[Sch55]{Schaffer1955}
J.~J. Sch\"{a}ffer.
\newblock On unitary dilations of contractions.
\newblock {\em Proc. Amer. Math. Soc.}, 6:322, 1955.

\bibitem[SN53]{Nagy1953}
B\'{e}la Sz.-Nagy.
\newblock Sur les contractions de l'espace de {H}ilbert.
\newblock {\em Acta Sci. Math. (Szeged)}, 15:87--92, 1953.

\bibitem[Sta15]{Starling2015}
Charles Starling.
\newblock Boundary quotients of {$\rm C^*$}-algebras of right {LCM} semigroups.
\newblock {\em J. Funct. Anal.}, 268(11):3326--3356, 2015.

\bibitem[Var74]{Varopoulos1974}
N.~Th. Varopoulos.
\newblock On an inequality of von {N}eumann and an application of the metric
  theory of tensor products to operators theory.
\newblock {\em J. Functional Analysis}, 16:83--100, 1974.

\end{thebibliography}

\end{document}